\documentclass[12pt,english]{smfart}
\usepackage{amsthm}

\usepackage[T1]{fontenc}
\usepackage[english,francais]{babel}

\newtheorem{theorem}{Theorem}[section]
\newtheorem{thm}{Theorem}[section]
\newtheorem{lemma}[theorem]{Lemma}

\newtheorem{rem}[thm]{Remark}
\newtheorem{conj}[thm]{Conjecture}
\begin{document}

\title{The number of solutions to $y^2=px(Ax^2+2)$}
\date{\today}
\author{Tarek GARICI}
\address{Facult\'e de Math\'emath\'ematiques, U.S.T.H.B., LA3C \\
Bp 32 El Alia 16111 Bab Ezzouar Alger, Alg\'erie.}
\email{tgarici@usthb.dz}
\author{Omar KIHEL}
\address{Department of Mathematics,Brock University, Ontario, Canada L2S 3A1\\}
\email{okihel@brocku.ca}
\author{Jesse LARONE}
\address{Department of Mathematics,Brock University, Ontario, Canada L2S 3A1\\}
\email{jl08yo@brocku.ca}

\keywords{Diophantine equations, Legendre symbol, Integral points on elliptic curves.}

\begin{abstract}
In this paper, we find a bound for the number of the positive solutions to the titled equation, improving a result of Togb\'e. As a consequence, we prove a conjecture of Togb\'e in a few cases. 
\end{abstract}
\maketitle

\section{Introduction}
Cassels \cite{Ca:1985} was challenged to determine when the sum of three consecutive cubes equals a square.  He \cite{Ca:1985} reduced the problem to finding integral points on the elliptic curve $y^2=3x(x^2+2)$.  Using the arithmetic of certain quartic number fields, he obtained that the integral points on the above elliptic curve were $(x,y)=(0,0)$, $(1,3)$, $(2,6)$, and $(24,204)$.

Using the classical work of Ljunggren \cite{Lj:1954} and its generalizations (see \cite{Ak:2009}, \cite{ChVo:1997}, \cite{Yu:1997}, and \cite{YuLi:2009}), Luca and Walsh \cite{LuWa:2001} considered the problem of finding the number of positive integer solutions to the Diophantine equation $y^2=nx(x^2+2)$, where $n>1$ is a positive integer.  They proved that the number of positive integer solutions to $y^2=nx(x^2+2)$ is at most $3\cdot 2^{\omega (n)}-1$, where $\omega (n)$ is the number of distinct prime factors of $n$.  In \cite{Ch:2010}, Chen considered the case where $n$ is a prime number greater than $3$.  He proved, in particular, that the Diophantine equation $y^2=nx(x^2+2)$ has at most two positive integer solutions.

Recently, Togb\'e \cite{To:2014} considered the more general Diophantine equation
\begin{equation}\label{eqM}
y^2=px(Ax^2+2),
\end{equation}
where $p$ is a prime number and $A$ is an odd integer greater than $1$.  He proved the following theorem.
\begin{thm}\label{To1}
For any prime $p$ and any odd positive integer $A>1$, the Diophantine equation \eqref{eqM} has at most seven positive integer solutions $(x,y)$.
\end{thm}
Using results obtained through MAGMA, he then made the following conjecture on sharp bounds for the number of solutions to equation \eqref{eqM}.
\begin{conj}\label{To2}
Let $p$ be a prime and $A>1$ any odd positive integer.
\begin{enumerate}
\item If $(A,p)\equiv (1,1)$, $(1,5)$, $(1,7)$, $(3,1)$, $(3,3)$, $(3,7)$, $(5,1)$, $(5,5)$, $(5,7)$, $(7,3)$, or $(7,5) \pmod{8}$, then Diophantine equation \eqref{eqM} has at most one positive integer solution $(x,y)$.
\item If $(A,p)\equiv (1,3)$ or $(7,1)$, then Diophantine equation \eqref{eqM} has at most two positive integer solutions $(x,y)$.
\item If $(A,p)\equiv (3,5)$ or $(7,7)$, then Diophantine equation \eqref{eqM} has at most three positive integer solutions $(x,y)$.
\end{enumerate}
\end{conj}
 
The aim of this paper is to improve the bound on the number of solutions to the Diophantine equation \eqref{eqM} provided in Theorem \ref{To1}, and to prove Conjecture \ref{To2} in some cases.  The main result of this paper is the following theorem.
\begin{thm}\label{th1}
Let $p$ be a prime and let $A>1$ be an odd integer.
\begin{enumerate}
\item If $p=2$, then Diophantine equation \eqref{eqM} has at most one positive integer solution $(x,y)$.
\item Suppose that $ p\mid A$ or $\left( \frac{-2A}{p}\right)= -1$, where $p$ is odd.
\begin{enumerate}
\item If $(A,p)\equiv (7,1)$ or $(7,7) \pmod{8}$, then Diophantine equation \eqref{eqM} has at most three positive integer solutions $(x,y)$.
\item Diophantine equation \eqref{eqM} has at most one positive integer solution $(x,y)$ otherwise.
\end{enumerate}

\item Suppose that $\left( \frac{-2A}{p}\right)=1$, where $p$ is odd. 
\begin{enumerate}
\item If $(A,p)\equiv (1,5)$, $(1,7)$, $(3,3)$, $(5,5)$, $(7,3)$, or $(7,5) \pmod{8}$, then Diophantine equation \eqref{eqM} has at most one positive integer  solution $(x,y)$.
\item If $(A,p)\equiv (1,1)$, $(3,1)$, $(3,7)$, $(5,1)$, $(5,3)$, or $(5,7) \pmod{8}$, then Diophantine equation \eqref{eqM} has at most two positive integer solutions $(x,y)$.
\item If $(A,p)\equiv (1,3)$ or $(3,5) \pmod{8}$, then Diophantine equation \eqref{eqM} has at most three positive integer solutions $(x,y)$.
\item If $(A,p)\equiv (7,7) \pmod{8}$, then Diophantine equation \eqref{eqM} has at most four positive integer solutions $(x,y)$.
\item If $(A,p)\equiv (7,1) \pmod{8}$, then Diophantine equation \eqref{eqM} has at most six positive integer solutions $(x,y)$.
\end{enumerate}
\end{enumerate}
\end{thm}

We will also prove the following result.
\begin{thm}\label{th5}
Let $p$ be a prime and let $A>1$ be an even integer.

\begin{enumerate}
\item If $p=2$, then Diophantine equation \eqref{eqM} has at most two positive integer solutions $(x,y)$. Moreover, if $A\equiv 0\pmod{4}$ and $ A\neq 2^{6}\cdot 1785$, then Diophantine equation \eqref{eqM} has at most one positive integer solution $(x,y)$.

\item Suppose that $p\mid A $ or $\left( \frac{-2A}{p}\right) =-1$, where $p$ is odd.

\begin{enumerate}
\item If $A\equiv 0\pmod{4}$, then Diophantine equation \eqref{eqM} has at most one positive integer solution $(x,y)$.

\item If $A\equiv 2\pmod{4}$, then Diophantine equation \eqref{eqM} has at most two positive integer solutions $(x,y)$.
\end{enumerate}

\item Suppose that $\left( \frac{-2A}{p}\right) =1$, where $p$ is odd. 

\begin{enumerate}
\item If $(A,p)\equiv (0,3)\pmod{4}$, then Diophantine equation \eqref{eqM} has at most one positive integer solution $(x,y)$.

\item If $(A,p)\equiv (0,1)\pmod{4}$, then Diophantine equation \eqref{eqM} has at most two positive integer solutions $(x,y)$.

\item If $(A,p)\equiv (2,3)\pmod{4}$, then Diophantine equation \eqref{eqM} has at most three positive integer solutions $(x,y)$.

\item If $(A,p)\equiv (2,1)\pmod{4}$, then Diophantine equation \eqref{eqM} has at most four positive integer solutions $(x,y)$.
\end{enumerate}
\end{enumerate}
\end{thm}

\section{Preliminary results}
We present the results required to prove Theorem \ref{th1} and Theorem \ref{th5}.  Recall that if $q$ is a prime number, $\nu_q (m)$ denotes the $q$-adic valuation of $m$.

Let $a$ and $b$ be odd positive integers for which the equation $aX^2-bY^2=2$ has a solution in positive integers $(X,Y)$.  Let $(a_1,b_1)$ be the minimal positive solution to this equation and define
\begin{equation*}
\alpha=\frac{a_1\sqrt{a}+b_1\sqrt{b}}{\sqrt{2}}.
\end{equation*}
For an odd integer $k$, define $a_k$ and $b_k$ by
\begin{equation*}
\alpha^k=\frac{a_k\sqrt{a}+b_k\sqrt{b}}{\sqrt{2}}.
\end{equation*}
Luca and Walsh proved the following result in \cite{LuWa:2001} regarding the solutions to the equation
\begin{equation}\label{eq0}
aX^2-bY^4=2.
\end{equation}
\begin{thm}\label{th2}
\text{ }
\begin{enumerate}
\item If $b_1$ is not a square, then equation \eqref{eq0} has no solution.
\item If $b_1$ is a square and $b_3$ is not a square, then $(X,Y)=(a_1,\sqrt{b_1})$ is the only solution to equation \eqref{eq0}.
\item If $b_1$ and $b_3$ are both squares, then $(X,Y)=(a_1,\sqrt{b_1})$ and $(a_3,\sqrt{b_3})$ are the only solutions to equation \eqref{eq0}.
\end{enumerate}
\end{thm}

Ljunggren proved the following result in \cite{Lj:1954}.
\begin{thm}\label{th3}
Let $a>1$ and $b$ be two positive integers.  The equation
\begin{equation*}
aX^2-bY^4=1
\end{equation*}
has at most one solution in positive integers $(X,Y)$.
\end{thm}

Let $D$ be a positive non-square integer, and let $\epsilon_D=T_1+U_1\sqrt{D}$ denote the minimal unit greater than 1, of norm 1, in $\mathbb{Z}[\sqrt{D}]$.  Define ${\epsilon_D}^k=T_k+U_k\sqrt{D}$ for $k \geq 1$.  Togb\'{e}, Voutier, and Walsh proved the following result in \cite{ToVoWa:2005}.
\begin{thm}\label{th4}
Let $D$ be a positive non-square integer.  There are at most two positive integer solutions $(X,Y)$ to the equation $X^2-DY^4=1$.  
\begin{enumerate}
\item If two solutions such that $Y_1<Y_2$ exist, then ${Y_1}^2=U_1$ and ${Y_2}^2=U_2$, except only if $D=1785$ or $D=16 \cdot 1785$, in which case ${Y_1}^2=U_1$ and ${Y_2}^2=U_4$.
\item If only one positive integer solution $(X,Y)$ to the equation $X^2-DY^4=1$ exists, then $Y^2=U_{\ell}$ where $U_1=\ell v^2$ for some square-free integer $\ell$, and either $\ell=1$, $\ell=2$, or $\ell=p$ for some prime $p \equiv 3 \pmod{4}$.
\end{enumerate}
\end{thm}
We will make the theorem \ref{th1} more precise when $D$ is even.
\begin{lemma}\label{lem7}
Let $D$ be a positive non-square integer.  Suppose that $D=2d$ where $d$ is a positive integer different from $8 \cdot 1785$.  Then the equation $X^2-DY^4=1$ has at most one positive solution $(X,Y)$.
\end{lemma}
\begin{proof}
Suppose that there exist two solutions to the equation $X^2-DY^4=1$.  Then there exist positive integer solutions $(X_1,Y_1)$ and $(X_2,Y_2)$ such that $Y_1<Y_2$.  It follows from Theorem \ref{th4} that ${Y_1}^2=U_1$, ${Y_2}^2=U_2$, and $U_2=2T_1U_1$, so ${Y_2}^2=2T_1{Y_1}^2$.  Then
\begin{equation}\label{eq2:2}
2\nu_2(Y_2)=1+\nu_2(T_1)+2\nu_2(Y_1).
\end{equation}
Since $\epsilon_D=T_1+U_1\sqrt{D}$ is a unit of norm $1$ in $\mathbb{Z}[\sqrt{D}]$ and $D=2d$, we obtain ${T_1}^2-2d{U_1}^2=1$, so that $T_1$ is odd.  Then $\nu_2(T_1)=0$, which is a contradiction with \eqref{eq2:2}. 
\end{proof}

\section{Main results}
{\it Proof of Theorem \ref{th1}.}  
Let $p=2$, and let $A$ be an odd positive integer.  Let $x,y$ be positive integers such that $y^2=2x(Ax^2+2)$. It is not difficult to see that $4$ divides $x$ and $y$.  Let $y=4w$ and $x=4z$.  Then we obtain
\begin{equation*}
w^2=z(8Az^2+1).
\end{equation*}
Since $\gcd(z,8Az^2+1)=1$, there exist positive integers $u$ and $v$ such that $z=u^2$, $8Az^2+1=v^2$, and
\begin{equation*}
v^2-8Au^4=1.
\end{equation*}
By Lemma \ref{lem7}, this equation has at most one positive integer solution $(u,v)$.

Let $p$ be an odd prime, and let $A$ be an odd positive integer.  Let $x,y$ be positive integers such that $y^2=px(Ax^2+2)$.  We remark that $\gcd(x,Ax^2+2)=1$ or $2$, so we consider two cases depending on the parity of $x$, with each case yielding two equations.  Suppose first that $x$ is even, so we let $x=2z$.  Since $p$ is prime, we let $y=2pw$.  Then we obtain
\begin{equation*}
pw^2=z(2Az^2+1).
\end{equation*}
Since $\gcd(z,2Az^2+1)=1$, there exist positive integers $u$ and $v$ such that either $z=pu^2$, $2Az^2+1=v^2$, and
\begin{equation}\label{eq1}
v^2-2Ap^2u^4=1,
\end{equation}
or $z=u^2$, $2Az^2+1=pv^2$, and
\begin{equation}\label{eq2}
pv^2-2Au^4=1.
\end{equation}
Suppose next that $x$ is odd.  Since $p$ is prime, we let $y=pw$.  Then we obtain
\begin{equation*}
pw^2=x(Ax^2+2).
\end{equation*}
Since $\gcd (x, Ax^2+2)=1$, there exist odd integers $u$ and $v$ such that either $x=pu^2$, $Ax^2+2=v^2$, and
\begin{equation}\label{eq3}
v^2-Ap^2u^4=2,
\end{equation}
or $x=u^2$, $Ax^2+2=pv^2$, and
\begin{equation}\label{eq4}
pv^2-Au^4=2.
\end{equation}
We consider each of the above four equations separately to determine the number of positive integer solutions to equation \eqref{eqM}.  

We begin with equation \eqref{eq1}.  Let $D=2Ap^2$.  By Lemma \ref{lem7}, equation \eqref{eq1} has at most one positive integer solution.

We next consider equation \eqref{eq2}, which has at most one positive integer solution by Theorem \ref{th3}.  It follows from this equation that $v$ is odd and that $u$ is even if and only if $p\equiv 1 \pmod{8}$.  If $p\equiv 3,5,$ or $7 \pmod{8}$, then $u$ is odd, and we obtain $p-2A\equiv 1 \pmod{8}$.  Then equation \eqref{eq2} has a solution only if $(A,p)\equiv (1,1)$, $(3,1)$, $(5,1)$, $(7,1)$, $(1,3)$, $(5,3)$, $(3,7)$, or $(7,7) \pmod{8}$.  Furthermore, equation \eqref{eq2} has a solution only if $\left( \frac{-2A}{p}\right) =1$.

Equation \eqref{eq3} has at most two positive integer solutions by Theorem \ref{th2}.  Since $u$ and $v$ are both odd, we have $1-A\equiv 2 \pmod{8}$ so $A\equiv 7 \pmod{8}$ and $v^2\equiv 2 \pmod{p}$ so $\left( \frac{2}{p} \right)=1$.  Then $p\equiv 1$ or $7 \pmod{8}$, and equation \eqref{eq3} has at least one solution only if $(A,p)\equiv (7,1)$ or $(7,7) \pmod{8}$.

Equation \eqref{eq4} has at most two positive integer solutions by Theorem \ref{th2}.  Since $u$ and $v$ are odd, we have $p-A\equiv 2 \pmod{8}$ so that equation \eqref{eq4} has a solution only if $(A,p)\equiv (1,3)$, $(3,5)$, $(5,7)$, or $(7,1) \pmod{8}$.  In particular, suppose that equation \eqref{eq4} has two solutions, and let $(a_1,b_1)$ be the minimal positive solution of $pX^2-AY^2=2$, so
\begin{equation*}
p{a_1}^2-A{b_1}^2=2.
\end{equation*}
Let
\begin{equation*}
\alpha=\frac{a_1\sqrt{p}+b_1\sqrt{A}}{\sqrt{2}},
\end{equation*}
and compute $\alpha^3$ to obtain
\begin{equation*}
b_3=\frac{3{a_1}^2pb_1+{b_1}^3A}{2}.
\end{equation*}
Since we assume that two solutions exist to equation \eqref{eq4}, $b_1$ and $b_3$ must both be squares by Theorem \ref{th2}.  It follows that there exist two positive integers $B_1$ and $B_3$ such that $b_1={B_1}^2$, $b_3={B_3}^2$, and
\begin{equation*}
3{a_1}^2p{B_1}^2+{B_1}^6A=2{B_3}^2.
\end{equation*}
This yields $\left( \frac{2}{p}\right)=\left( \frac{A}{p}\right)$.  Since $-Au^4\equiv 2 \pmod{p}$, we obtain $\left( \frac{-A}{p}\right)=\left( \frac{2}{p}\right)=\left( \frac{A}{p}\right)$ so $\left( \frac{-1}{p} \right)=1$.  It follows that $p\equiv 1 \pmod{4}$, so $p \equiv 1$ or $5 \pmod{8}$.  Therefore equation \eqref{eq4} has at most two positive integer solutions only if $(A,p)\equiv (3,5)$ or $(7,1) \pmod{8}$, and it has at most one positive integer solution only if $(A,p)\equiv (1,3)$ or $(5,7) \pmod{8}$.
Furthermore, equation \eqref{eq4} has a solution only if $\left( \frac{-2A}{p}\right) =1$.\\
Since the number of solutions to equations \eqref{eq2} and \eqref{eq4} depends on the value of $\left( \frac{-2A}{p}\right)$, we first suppose that $p\mid A$ or $\left( \frac{-2A}{p}\right)= -1$, then equations \eqref{eq2} and \eqref{eq4}  have no integer solution, equation \eqref{eq1} has at most one solution, and \eqref{eq3} has at most two positive integer solutions only if $(A,p)\equiv (7,1)$, or $(7,7) \pmod{8}$. Therefore when $p\mid A$ or $\left( \frac{-2A}{p}\right)=-1$, equation \eqref{eqM} has at most three positive integer solutions if $(A,p)\equiv (7,1)$, or $(7,7) \pmod{8}$, and it has at most one positive integer solution in all other cases.\\

We next suppose that $\left( \frac{-2A}{p}\right)=1$.  Then equation \eqref{eq2} has at most one positive integer solution.\\

\noindent If $A\equiv 1 \pmod{8}$, then equation \eqref{eq1} has at most one solution, \eqref{eq2} has at most one solution and only if $p \equiv 1$ or $3 \pmod{8}$, \eqref{eq3} has no solution, and \eqref{eq4} has at most one solution and only if $p \equiv 3 \pmod{8}$.\\

\noindent If $A\equiv 3 \pmod{8}$, then equation \eqref{eq1} has at most one solution, \eqref{eq2} has at most one solution and only if $p \equiv 1$ or $7 \pmod{8}$, \eqref{eq3} has no solution, and \eqref{eq4} has at most two solutions and only if $p \equiv 5 \pmod{8}$.\\

\noindent If $A\equiv 5 \pmod{8}$, then equation \eqref{eq1} has at most one solution, \eqref{eq2} has at most one solution and only if $p \equiv 1$ or $3 \pmod{8}$, \eqref{eq3} has no solution, and \eqref{eq4} has at most one solution and only if $p \equiv 7 \pmod{8}$.\\

\noindent If $A\equiv 7 \pmod{8}$, then equation \eqref{eq1} has at most one solution, \eqref{eq2} has at most one solution and only if $p \equiv 1$ or $7 \pmod{8}$, \eqref{eq3} has at most two solutions and only if $p \equiv 1$ or $7 \pmod{8}$, and \eqref{eq4} has at most one solution and only if $p \equiv 1 \pmod{8}$. \qed \\

\noindent {\it Proof of Theorem \ref{th5}.} If $A$ is even and $p$ is odd, we let $A=2A'$. Then 
\begin{equation*}
y^{2}=2px(A'x^{2}+1).
\end{equation*}%
We let $y=2pw$, and we obtain 
\begin{equation*}
2pw^{2}=x(A'x^{2}+1).
\end{equation*}%
Since $\gcd (x,A'x^{2}+1)=1,$ there exist positive integers $u$ and 
$v$ such that either \newline
$x=2pu^{2}$, $A'x^{2}+1=v^{2}$, and 
\begin{equation}
v^{2}-4A'p^{2}u^{4}=1,  \label{eq6}
\end{equation}%
or $x=2u^{2},$ $A'x^{2}+1=pv^{2}$ and 
\begin{equation}
pv^{2}-4A'u^{4}=1,  \label{eq5}
\end{equation}%
or $x=u^{2},$ $A'x^{2}+1=2pv^{2}$ and 
\begin{equation}
2pv^{2}-A'u^{4}=1,  \label{eq7}
\end{equation}%
or $x=pu^{2},$ $A'x^{2}+1=2v^{2}$ and 
\begin{equation}
2v^{2}-A'p^{2}u^{4}=1.  \label{eq8}
\end{equation}
If $A'$ is a perfect square, then equation \eqref{eq6} has no positive integer solution, otherwise it has at most one positive
integer solution by Lemma \ref{lem7}.\newline
By Theorem \ref{th3}, each of equations \eqref{eq5}, \eqref{eq7}, and \eqref{eq8} has at most one solution. Equation \eqref{eq5} has a solution only if $p\equiv 1\pmod{4}$ and $\left( \frac{-A'}{p}\right) =1$, equation \eqref{eq7} has a solution only if $A'$ is odd and $\left( \frac{-A'}{p}\right) =1,$ and equation \eqref{eq8} has a solution only if $A'$ is odd.
Since the number of solutions to equations \eqref{eq5} and \eqref{eq7} depends on the value of $\left( \frac{-A'}{p}\right) =\left( \frac{
-2A}{p}\right) $, we first suppose that $p\mid A$ or $\left( \frac{-2A}{p}\right)=-1$.  Then equations \eqref{eq5} and \eqref{eq7} have no integer solution.\\

\noindent If $A\equiv 0\pmod{4}$, then equation \eqref{eq6} has at most one solution, \eqref{eq5} has no solution, \eqref{eq7} 
has no solution, and \eqref{eq8} has no solution.\newline

\noindent If $A\equiv 2\pmod{4}$, then equation \eqref{eq6} has at most one solution, \eqref{eq5} has no solution, \eqref{eq7} has no solution, and \eqref{eq8} has at most one solution.\\

We now suppose that $\left( \frac{-2A}{p}\right) =1$. Then equations \eqref{eq5} and \eqref{eq7} have at most one positive integer solution. 

\noindent If $A\equiv 0\pmod{4}$, then equation \eqref{eq6} has at most one solution, \eqref{eq5} has at most one solution only if $p\equiv 1\pmod{4}$, \eqref{eq7} has no solution, and \eqref{eq8} has no solution.\\

\noindent If $A\equiv 2\pmod{4}$, then equation \eqref{eq6} has at most one solution, \eqref{eq5} has at most one solution only if $p\equiv 1\pmod{4}$, \eqref{eq7} has at most one solution, and \eqref{eq8} has at most one solution.\\

If $A$ is even and $p=2$, we let $A=2A'$. Then 
\begin{equation*}
y^{2}=2x(2A'x^{2}+2).
\end{equation*}%
We let $y=2w$, and we obtain 
\begin{equation*}
w^{2}=x(A'x^{2}+1).
\end{equation*}%
Since $\gcd (x,A'x^{2}+1)=1$, there exist positive integers $u$ and $v$ such that $x=u^{2}$, $A'x^{2}+1=v^{2}$, and 
\begin{equation}
v^{2}-A'u^{4}=1,  \label{eq9}
\end{equation}%
which has no solution if $A'$ is a perfect square and at most two solutions by Theorem \ref{th4}. Moreover, if $A'$ is even and $%
A'\neq 2^{5}\cdot 1785,$ then by Lemma \ref{lem7} equation \eqref{eq9} has at most one solution. \qed

\begin{rem}
When we had finished writing the paper, we noticed that a proof of the result stated in Lemma \ref{lem7} already existed within the proof of Theorem 1 by Luca and Walsh in \cite{LuWa:2005}. Our proof of Lemma \ref{lem7} seems to be different than the proof of the result in \cite{LuWa:2005}.
\end{rem}

\begin{rem}
Theorem \ref{th1} implies that Conjecture \ref{To2} is true if $(A,p) \equiv (1,5)$, $(1,7)$, $(3,3)$, $(3,5)$, $(5,3)$, $(7,3)$, or $(7,5) \pmod{8}$.
\end{rem}


\end{document}